\documentclass[11pt,reqno]{amsart}
\usepackage{esint}
\usepackage{enumerate}
\usepackage{enumitem}
\usepackage{mathrsfs}
\usepackage{url}
\usepackage{cite}
\usepackage{comment}
\usepackage{amsfonts, amssymb, amscd, amsthm, bm, cancel,amsmath}
\usepackage{amsrefs}
\usepackage[
linktocpage=true,colorlinks,citecolor=magenta,linkcolor=blue,urlcolor=magenta]{hyperref}
\newtheorem{proposition}{Proposition}[section]
\newtheorem{theorem}[proposition]{Theorem}
\newtheorem{lemma}[proposition]{Lemma}
\newtheorem{corollary}[proposition]{Corollary}

\newtheorem{assumption}[proposition]{Assumption}
\theoremstyle{remark}
\newtheorem{remark}[proposition]{Remark}

\numberwithin{equation}{section}
\allowdisplaybreaks

\begin{document}
\title[Contraction of hypersurfaces with positive sectional curvature]{Contraction of hypersurfaces with positive sectional curvature in hyperbolic space}

\address{School of Mathematical Sciences, University of Science and Technology of China, Hefei 230026, P.R. China}
\author[T. Luo]{Tianci Luo}
\email{\href{mailto:Luo_tianci@mail.ustc.edu.cn}{Luo\_tianci@mail.ustc.edu.cn}}
\author[Y. Wei]{Yong Wei}
\email{\href{mailto:yongwei@ustc.edu.cn}{yongwei@ustc.edu.cn}}
\author[R. Zhou]{Rong Zhou}
\email{\href{mailto:zhourong@mail.ustc.edu.cn}{zhourong@mail.ustc.edu.cn}}

\date{\today}

\begin{abstract}
	We study contracting curvature flows of compact hypersurfaces with positive sectional curvature in hyperbolic space $\mathbb{H}^{n+1}$. The speed is assumed to be homogeneous of degree one in the principal curvatures and to satisfy certain conditions. This class of flows includes the $k$th mean curvature flow as a special case. We show that if the initial hypersurface has positive sectional curvature, then this property is preserved along the flow, and the evolving hypersurface contracts to a round point in finite time. 
\end{abstract}

\maketitle
\tableofcontents

\section{Introduction}

Let $X_0:M^n\to\mathbb{H}^{n+1}$ be a smooth embedding such that $M_0=X_0(M)$ is a closed smooth hypersurface in the hyperbolic space $\mathbb{H}^{n+1}$. We consider the contracting curvature flow in hyperbolic space, which is a family of smooth immersions $X:M^n \times [0,T)\to \mathbb{H}^{n+1}$ satisfying the evolution equation
\begin{equation}
	\label{equ-f}
	\begin{cases}
		\begin{aligned}
			\dfrac{\partial}{\partial t}X(x,t) & = -F(x,t)\,\nu(x,t), \\
			X(\cdot,0) & = X_0,
		\end{aligned}
	\end{cases}
\end{equation}
where $\nu(x,t)$ is the unit outward normal of $M_t:=X(M,t)$, and $F$ is a smooth symmetric function of the principal curvatures $\kappa=(\kappa_1,\cdots,\kappa_n)$ of the hypersurfaces.  When $F=H$ is the mean curvature, the flow \eqref{equ-f} reduces to the mean curvature flow. 

There is a substantial literature on contracting curvature flows \eqref{equ-f} for hypersurfaces in Euclidean space. In the foundational work \cite{Hui86}, Huisken proved that any compact strictly convex hypersurface in Euclidean space evolving by mean curvature flow becomes spherical as it shrinks to a point. This result was later generalized by Chow to the flow by the $n$th root of the Gauss curvature \cite{chow85} and by the square root of the scalar curvature \cite{chow87}, although in the latter case an additional assumption beyond convexity is required. Andrews \cite{And94,And07} considered a broad class of speed functions $F$ that are homogeneous of degree one in the principal curvatures and satisfy natural concavity conditions. In particular, it is proved that the $k$th mean curvature flow
\begin{equation}\label{s1.kmcf}
\frac{\partial }{\partial t}X=-E_k^{1/k}\nu,\qquad k=1,\cdots,n    
\end{equation}
evolves any strictly convex initial hypersurface in Euclidean space to a round point, where $E_k=\binom{n}{k}^{-1}\sigma_k(\kappa)$ denotes the normalized $k$th mean curvature of the hypersurface. The special case of surfaces in $\mathbb{R}^3$ was studied in \cite{And10,schn05}. For speed functions $F$ with higher homogeneity, the analysis becomes much more delicate. The case that is best understood so far is the flow by powers of the Gauss curvature, $F=K^{\alpha}$ with $\alpha\geq \frac{1}{n+2}$; see \cite{And96,And99,And00,AGN16,BCD17,CD16,GN17,Tso85}. For flows by powers of general $k$th mean curvature, a strong curvature pinching condition on the initial hypersurface is typically required to ensure that a closed convex hypersurface shrinks to a round point \cite{AS10,AM12,Sch05,Sch06}.

For curvature flows \eqref{equ-f} in the sphere $\mathbb{S}^{n+1}$, Huisken \cite{Hui87} proved that for any initial hypersurface satisfying a certain quadratic curvature pinching condition, the mean curvature flow either contracts the hypersurface to a point in finite time or evolves for all time to a smooth totally geodesic hypersurface. This quadratic pinching condition was refined recently in \cite{lan25}.  Gerhardt \cite{Ger15} showed that if the speed function $F$ is concave and inverse concave with respect to the principal curvatures, the flow \eqref{equ-f} contracts any strictly convex hypersurface in the sphere to a round point in finite time. This includes the $k$th mean curvature flow \eqref{s1.kmcf} as a special case. The two dimensional case of the flow in $\mathbb{S}^3$ has also been studied by Andrews \cite{And02},  McCoy \cite{McC18} and Hu-Li-Wei-Zhou \cite{HLWZ20}.

The understanding of the flow \eqref{equ-f} in hyperbolic space $\mathbb{H}^{n+1}$ remains less complete. Besides strict convexity, there are several other convexity conditions for hypersurfaces in hyperbolic space, defined pointwise in terms of the principal curvatures $\kappa_i$:
\begin{itemize}
\item $h$-convexity, if $\kappa_i>1$ for all $i=1,\cdots,n$;
\item $\kappa_iH>n$ for all $i=1,\cdots,n$;
\item positive sectional curvature, if $\kappa_i\kappa_j>1$ for all $i\neq j$;
\item positive Ricci curvature, if $\kappa_i(H-\kappa_i)>n-1$ for all $i=1,\dots,n$;
\item strict convexity, if $\kappa_i>0$ for all $i=1,\cdots,n$.
\end{itemize}
Many results on curvature flows in hyperbolic space have been obtained under the strongest assumption of $h$-convexity. Andrews \cite{And9402} considered a broad class of fully nonlinear flows, excluding the mean curvature flow, and proved that any $h$-convex initial hypersurface shrinks to a round point. Yu \cite{Yu16} studied the flow \eqref{equ-f} for $h$-convex hypersurfaces in hyperbolic space, for a general class of speed functions homogeneous of degree one, using arguments similar to those employed by Gerhardt \cite{Ger15} for curvature flows in the sphere. For the mean curvature flow, however, weaker assumptions are sufficient. Huisken \cite{Hui8602} proved that a compact hypersurface in $\mathbb{H}^{n+1}$ contracts to a round point, provided that its principal curvatures satisfy the condition $\kappa_iH>n$ for all $i=1,\dots,n$. This condition was later weakened to positive Ricci curvature by Andrews and Chen \cite{AC17}. In the special case of surfaces in $\mathbb{H}^3$, Andrews and Chen \cite{AC17} also studied the flow \eqref{equ-f} for surfaces with positive scalar curvature $R=2(K-1)>0$, and showed that both the mean curvature flow and the scalar curvature flow deform such surfaces to a round point. The flows by powers of the mean curvature and scalar curvature for surfaces with positive scalar curvature in $\mathbb{H}^3$ were further studied by Hu-Li-Wei-Zhou \cite{HLWZ20}. For flows with other speed, Xu \cite{xu10} showed that harmonic mean curvature flow deforms strictly convex hypersurfaces in the hyperbolic  space to a round point.  Chen and Huang \cite{CH24}  proved that the flow by powers of the Gauss curvature with $\alpha>\frac{1}{n+2}$ deforms strictly convex hypersurfaces in the sphere or hyperbolic space to a round point.

Recently, Andrews, Chen, and the second author \cite{ACW21} studied the volume-preserving curvature flows in hyperbolic space. They proved that if the initial hypersurface has positive sectional curvature, then the positivity of sectional curvature is preserved along a large class of volume preserving flows, including the volume preserving $k$th mean curvature flow for all $k=1,\cdots,n$. As an application, they established certain Alexandrov–Fenchel inequalities for hypersurfaces with positive sectional curvature in hyperbolic space. 

Motivated by these works, in this paper we study the contracting curvature flow \eqref{equ-f} for hypersurfaces with positive sectional curvature in hyperbolic space $\mathbb{H}^{n+1}$. Before we state our result, we impose the following assumptions on the function $F$.

\begin{assumption}	\label{assumption}
$F:\Gamma_+\to (0,\infty)$ is a smooth symmetric function defined on the positive cone
	\begin{equation*}
		\Gamma_+ := \{ \kappa=(\kappa_1,\kappa_2,\cdots,\kappa_n)\in\mathbb{R}^n:\kappa_i>0,\ \forall\ i=1,2,\cdots,n \},
	\end{equation*}
	and satisfies the following properties:
	\begin{enumerate}
		\item[{\rm (i)}] $F$ is strictly increasing in each variable, homogeneous of degree one, and normalized by $F(1,1,\cdots,1)=1$;
		\item[{\rm (ii)}] For all $(y_1,y_2,\dots,y_n)\in\mathbb{R}^n$,
		\begin{equation}
			\sum\limits_{k,\ell}\ddot{F}^{k\ell}y_ky_{\ell} \geq F^{-1} \left(\sum\limits_{k=1}^n \dot{F}^k y_k \right)^2 - \sum\limits_{k=1}^n \dfrac{\dot{F}^k}{\kappa_k}y_k^2,
			\label{equ-property3}
		\end{equation}
        where $\dot{F}^k=\frac{\partial F}{\partial \kappa_k}$ and $\ddot{F}^{k\ell}=\frac{\partial^2F}{\partial\kappa_k\partial\kappa_\ell}$. 
	\end{enumerate}
\end{assumption}

In section \ref{sec.3}, we will describe the properties of functions satisfying Assumption \ref{assumption} in more details. In particular, we will see that such functions must be invere concave (Lemma \ref{lem:inverse-concave}); the condition \eqref{equ-property3} is equivalent to that $\log F$ is convex as a function of $(\log \kappa_1,\cdots,\log\kappa_n)$ (Lemma \ref{lem:logF-convex}); Using the log-convexity property, we can derive a useful lower bound $F\geq \left(\prod_{i=1}^n \kappa_i\right)^{1/n}$ (Proposition \ref{lem:gm-lower-bound}).  
\begin{remark}
Examples of functions satisfying Assumption \ref{assumption} include 
\[F=\left(\frac{1}{n}\sum_{i=1}^n \kappa_i^r\right)^{\frac{1}{r}},\quad r>0\]
and 
\[F=E_k^{\frac{1}{k}},\quad k=1,2,\dots,n\] 
where $E_k$ is the normalized $k$th elementary symmetric function defined by
\begin{equation*}
			E_k(\kappa) = {n\choose k}^{-1}\sigma_k(\kappa) = {n\choose k}^{-1}\sum\limits_{1\leq i_1 <\cdots < i_k \leq n} \kappa_{i_1}\kappa_{i_2}\cdots\kappa_{i_k}.
\end{equation*}
Moreover, assume that $F^{(1)},\dots,F^{(m)}$ satisfy Assumption \ref{assumption}, and let $\theta_\alpha>0$ and $\sum_{\alpha=1}^m \theta_\alpha=1$.Then
\[
F(\kappa):=\prod_{\alpha=1}^m \big(F^{(\alpha)}(\kappa)\big)^{\theta_\alpha}
\]
also satisfies Assumption \ref{assumption}.
\end{remark}

The main result of this paper is the following convergence theorem for the flow \eqref{equ-f}.

\begin{theorem}\label{thm-main}
	Let $X_0:M^n \to \mathbb{H}^{n+1}$ be a smooth, closed hypersurface in hyperbolic space $\mathbb{H}^{n+1}$ ($n\geq 3$) with positive sectional curvature. Assume that $F$ satisfies Assumption \ref{assumption}. Then the flow \eqref{equ-f} has a unique smooth solution $M_t$ on a maximal time interval $[0,T)$ with $T<\infty$. Moreover, $M_t$ has positive sectional curvature for every $t\in[0,T)$, and converges smoothly to a point $p\in\mathbb{H}^{n+1}$ as $t\to T$.
    
	If, in addition, $F$ is strictly concave, or $F=\frac{1}{n}H$, then the solution is asymptotic to a shrinking geodesic sphere as $t\to T$. More precisely, let $\Theta(t,T)$ denote the spherical solution of \eqref{equ-f} with extinction time $T$. Using geodesic polar coordinates centered at $p$, write $M_t$ as the graph of a function $u(\cdot,t)$ over $\mathbb{S}^n$. Then the rescaled function $u(\cdot,t)\Theta(t,T)^{-1}$ converges to $1$ in $C^\infty(\mathbb{S}^n)$ as $t\to T$. 
\end{theorem}

It is well known that $E_k^{\frac{1}{k}}(\kappa)$ for $2\leq k\leq n$ and $\bigl( \frac{1}{n}\sum_{i=1}^n \kappa_i^r \bigr)^{\frac{1}{r}}$ for $0< r<1$ are both strictly concave in $\Gamma_+$ (see \cite[Theorem 3.2]{Ger15} and \cite[Lemma 3.2]{Yu16}) and hence satisfy all conditions of Theorem \ref{thm-main}. Moreover, if $F_1$ and $F_2$ are two symmetric functions that are strictly concave, then $F_1^{\alpha_1}F_2^{\alpha_2}$ with $\alpha_1+\alpha_2=1$, $\alpha_1,\alpha_2>0$ is also strictly concave. Consequently, many curvature functions satisfy the conditions of Theorem \ref{thm-main}. Our result generalizes Andrews’ \cite{And07} classical theorem for contracting curvature flows in Euclidean space and Gerhardt’s \cite{Ger15} corresponding result in the sphere. 
		
We state the result for the $k$th mean curvature flow separately as follows.
\begin{corollary}
Let $X_0:M^n \to \mathbb{H}^{n+1}$ be a smooth closed hypersurface in the hyperbolic space $\mathbb{H}^{n+1}$ ($n\geq 3$) with positive sectional curvature. Then there exists a unique solution $M_t$ to the $k$th mean curvature flow
\[\frac{\partial }{\partial t}X=-E_k^{1/k}\nu,\qquad k=1,\cdots,n  \]
on a maximal time interval $[0,T)$ with $T<\infty$. Moreover, $M_t$ has positive sectional curvature for each $t\in[0,T)$, and converges smoothly to a round point $p\in\mathbb{H}^{n+1}$ as $t\to T$. 
\end{corollary}

\begin{remark}
When $n=2$, positivity of sectional curvature is equivalent to positivity of the scalar curvature, namely, $K=\kappa_1\kappa_2>1$. In this case, the corresponding results for surfaces with positive scalar curvature in $\mathbb{H}^3$ were established for the mean curvature flow ($k=1$) by Andrews-Chen \cite{AC17}, and for the Gauss curvature flow ($k=2$) by Hu-Li-Wei-Zhou \cite{HLWZ20}.
\end{remark}

To prove our main theorem, we establish a key curvature pinching estimate (Proposition \ref{s4.lem-pinc}). Motivated by \cite{ACW21}, we introduce a function on the orthonormal frame bundle $O(M)$ over $M$: given a point $x\in M$, $t\geq 0$, and an orthonormal frame $\mathbb{O}=\{e_1,\dots,e_n\}$ for $T_x M$ with respect to $g(x,t)$, we define
\begin{equation*}
	G(x,t,\mathbb{O}) = h_{(x,t)}(e_1,e_1)h_{(x,t)}(e_2,e_2) - h_{(x,t)}(e_1,e_2)^2 - 1 - \varepsilon F(x,t)^2.
\end{equation*}
Since the initial hypersurface has positive sectional curvature, one can choose $\varepsilon>0$ sufficiently small so that $G$ is positive initially. To apply the maximum principle and show that $G$ remains positive, we compute the time and spatial derivatives of $G$. This involves a rather delicate computation: using Hessian inequalities on the total space of $O(M)$, we establish the required differential inequality for the time derivative at a minimum point. The argument is closely related to that used by Andrews in proving a generalized tensor maximum principle in \cite{And07}, and also to the vector bundle maximum principles for reaction–diffusion equations by Hamilton in \cite{Ha86}. By further observing that $F\geq \left(\prod_{i=1}^n \kappa_i\right)^{1/n}$ for $F$ satisfying Assumption \ref{assumption} (see Proposition \ref{lem:gm-lower-bound}), we obtain the curvature pinching estimate 
\[\kappa_n\leq C\kappa_1\] 
along the flow \eqref{equ-f}. Once this curvature pinching estimate is established, we apply Stampacchia iteration, following Huisken \cite{Hui86} and Gerhardt \cite{Ger15}, to prove that the  flow satisfies a roundness estimate. This in turn implies that the rescaled flow converges smoothly to a geodesic sphere.

The paper is organized as follows. In Section \ref{sec2}, we collect preliminaries needed for the proof, including properties of symmetric curvature functions, the geometry of hypersurfaces in hyperbolic space $\mathbb{H}^{n+1}$, and evolution equations along the flow \eqref{equ-f}. In Section \ref{sec.3},  we describe the properties of symmetric functions satisfying Assumption \ref{assumption}.  In Section \ref{sec4}, we apply the maximum principle to prove the crucial curvature pinching estimate along the flow \eqref{equ-f}. In Section \ref{sec5}, we show that the flow contracts to a point in finite time, and a properly rescaled solution converges smoothly to a geodesic sphere.

\section{Preliminaries}
\label{sec2}
In this section, we recall properties of symmetric curvature functions, the geometry of hypersurfaces in the hyperbolic space $\mathbb{H}^{n+1}$, and evolution equations along the flow \eqref{equ-f}.

\subsection{Symmetric functions}\label{sec.2-1}
For a smooth symmetric function $F$ on $\mathbb{R}^n$, we may extend it to a smooth symmetric function on the space of symmetric matrices by setting $F(A)=F(\kappa(A))$, and we use the same notation $F$ for this extension. Here $A=(A_{ij})\in\mathrm{Sym}(n)$ is a symmetric matrix, and $\kappa(A)=(\kappa_1,\kappa_2,\cdots,\kappa_n)$ denotes the eigenvalues of $A$. We write $\dot{F}^{ij}(A)$ and $\ddot{F}^{ij,k\ell}(A)$ for the first and second derivatives of $F$ with respect to the components of its matrix argument, so that
\begin{equation*}
	\left.\dfrac{\partial}{\partial s}F(A+sB)\right|_{s=0} = \dot{F}^{ij}(A)B_{ij},\ \ \left.\dfrac{\partial^2}{\partial s^2}F(A+sB)\right|_{s=0} = \ddot{F}^{ij,k\ell}(A)B_{ij}B_{k\ell}
\end{equation*}
for any two symmetric matrices $A$, $B$. We also use the notation
\begin{equation*}
	\dot{F}^i(\kappa) = \dfrac{\partial F}{\partial \kappa_i}(\kappa),\ \ \ddot{F}^{ij}(\kappa) = \dfrac{\partial^2 F}{\partial \kappa_i \partial \kappa_j}(\kappa)
\end{equation*}
for the first and second derivatives of $F$ with respect to $\kappa$, evaluated at the eigenvalues of $A$. If $A$ is diagonal with distinct eigenvalues $\kappa=\kappa(A)$, then the first derivative of $F$ satisfies
\begin{equation*}
	\dot{F}^{ij}(A) = \dot{F}^i(\kappa)\delta_{ij},
\end{equation*}
and the second derivative of $F$ in the direction $B\in\mathrm{Sym}(n)$ can be expressed in terms of $\dot{F}$ and $\ddot{F}$ as follows (see \cite[Theorem 5.1]{And07} and \cite[Lemma 1.1]{ge96}):
\begin{equation}
	\ddot{F}^{ij,k\ell}(A)B_{ij}B_{k\ell} = \sum\limits_{i,j}\ddot{F}^{ij}(\kappa)B_{ii}B_{jj} + 2 \sum\limits_{i>j} \dfrac{\dot{F}^i(\kappa)-\dot{F}^j(\kappa)}{\kappa_i-\kappa_j}B_{ij}^2.
	\label{equ-secf}
\end{equation}
This formula remains valid, by continuity, in the case where some of the eigenvalues $\kappa_i$ coincide.

Let $F$ be a smooth, symmetric, homogeneous of degree one, monotone and concave function on $\mathbb{R}^n$. We say that $F$ strictly concave, if 
\[\ddot{F}^{ij}\xi_i\xi_j<0,\quad \forall~\xi\nsim \kappa,~\xi\neq 0,\]
or equivalently, if the multiplicity of the eigenvalue $\kappa=0$ for $D^2F(\kappa)$ is one for all $\kappa$. Note that the one-homogeneity of $F$ implies that $\kappa$ is an eigenvector of $D^2F(\kappa)$ corresponding to the eigenvalue zero. Examples of strictly concave symmetric functions include 
\[E_k^{1/k},\quad k=2,\cdots,n\qquad \text{and}\quad \left(\frac{1}{n}\sum_{i=1}^{n}\kappa_i^r\right)^{\frac{1}{r}},\quad -1\leq r<1.\]
The proof can be found in \cite[Theorem 3.2]{Ger15} and \cite[Lemma 3.2]{Yu16}) respectively. 


Another important notion for curvature functions is inverse concavity. We say that $F$ is inverse concave, if its dual function
\[
F_*(\mu_1,\dots,\mu_n):=\frac{1}{F(\mu_1^{-1},\dots,\mu_n^{-1})}
\]
is concave on $\Gamma_+$. We have the following useful estimate on inverse concave functions.
\begin{lemma}[see {\cite[Lemma 5]{AMZ13}}] Assume that $F=F(\kappa)$ is homogeneous of degree one and normalized by $F(1,\cdots,1)=1$.  If $F$ is inverse concave, then
	\begin{equation}\label{equ-inv2}
			\sum\limits_{i=1}^n \dot{F}^i \kappa_i^2 \geq F^2.
	\end{equation}
\end{lemma}

\subsection{Hypersurfaces in the hyperbolic space}
We view $\mathbb{H}^{n+1}=\mathbb{S}^n\times[0,\infty)$ as a warped product manifold equipped with the metric
\begin{equation*}
	\overline{g} = \mathrm{d}r^2 + \sinh^2 r g_{\mathbb{S}^n},
\end{equation*}
where $r\in[0,\infty)$ is the geodesic distance from the origin point and $g_{\mathbb{S}^n}$ denotes the standard metric on $\mathbb{S}^n$.

Let $M$ be a smooth closed hypersurface in $\mathbb{H}^{n+1}$ with induced metric $g$ and outward unit normal $\nu$. Denote by $D$ and $\nabla$ the Levi-Civita connections of the ambient metric $\overline{g}$ and the induced metric $g$, respectively. In local coordinates $\{x^1,\cdots,x^n\}$ on $M$, the induced metric and the second fundamental form are given by 
\[g_{ij} = \overline{g}(\partial_i,\partial_j),\qquad h_{ij} = \langle D_{\partial_i}\nu,\partial_j\rangle.\]
Here $h=(h_{ij})$ is a symmetric $(0,2)$-tensor. The associated Weingarten map is the $(1,1)$-tensor 
\[\mathcal{W} = (h^i_j),\qquad h^i_j = g^{ik}h_{kj}.\]  
Although the matrix \((h^i_j)\) is not in general symmetric in arbitrary coordinates, \(\mathcal W\) is self-adjoint with respect to \(g\). Hence \(\mathcal W\) is diagonalizable and has real eigenvalues
$\kappa=(\kappa_1,\dots,\kappa_n)$, which are the principal curvatures of \(M\). We also write $|A|^2=h^i_j h^j_i=g^{ip}g^{jq}h_{ij}h_{pq}$ for the squared norm of the second fundamental form.

Given a smooth symmetric  function \(F(\kappa)\), we can define the corresponding curvature function on \(M\) by
\[
F(\mathcal W)=F(\kappa(\mathcal W)).
\]
Equivalently, since \(\mathcal W=g^{-1}h\) is \(g\)-self-adjoint, the matrix \(g^{-1/2}hg^{-1/2}\in \mathrm{Sym}(n)\)
is symmetric and has the same eigenvalues as \(\mathcal W\). Thus \(F\) may be viewed as an \(O(n)\)-invariant smooth function of the symmetric matrix \(g^{-1/2}hg^{-1/2}\), so the general facts for symmetric functions described above apply directly. In particular, in a local orthonormal frame, where \(g_{ij}=\delta_{ij}\), the matrix of \(\mathcal W\) is simply \((h_{ij})\), which is symmetric. We denote by \(\dot F^{ij}\) and \(\ddot F^{ij,kl}\) the first and second derivatives of \(F\) with respect to the entries of \(h_{ij}\) in such a frame.

\subsection{Evolution equations}
Along the flow \eqref{equ-f} for hypersurfaces in hyperbolic space, the following evolution equations for geometric quantities on the evolving hypersurface $M_t$ are well known. For further details, we refer the readers to \cite{And9402,AC17,ACW21} and \cite{ge06}.
\begin{lemma}
Assume that the curvature function $F$ is homogeneous of degree one in the principal curvatures. 	Along the flow \eqref{equ-f} in hyperbolic space, the induced metric $g_{ij}$, the  curvature function $F$, and the second fundamental form $h_{ij}$ satisfy the following equations:
\begin{align}
    \frac{\partial}{\partial t}g_{ij}=&-2Fh_{ij},\label{s2.evl-g}\\
		\frac{\partial}{\partial t} F = & \dot{F}^{k\ell}\nabla_k\nabla_{\ell}F + \left( \dot{F}^{ij}(h^2)_{ij}- \dot{F}^{ij}g_{ij} \right)F,  \label{equ-psi}\\
        \frac{\partial}{\partial t}h_{ij}=&\nabla_i\nabla_jF-F\left((h^2)_{ij}+g_{ij}\right),\label{s2.evl-h1}\\
        \frac{\partial}{\partial t}h_{ij}=&\dot{F}^{k\ell}\nabla_k\nabla_{\ell}h_{ij} + \ddot{F}^{k\ell,pq}\nabla_i h_{k\ell} \nabla_j h_{pq} \nonumber\\
        &\quad + \left( \dot{F}^{k\ell}(h^2)_{k\ell} + \dot{F}^{k\ell}g_{k\ell}\right)h_{ij} - 2F\left((h^2)_{ij}+ g_{ij}\right),  \label{s2.evlh2}
		\end{align}
where $(h^2)_{ij}=h_{ik}h_{\ell j}g^{k\ell}$. 
\end{lemma}


\section{Functions satisfying Assumption \ref{assumption}}\label{sec.3}

In this section, we describe the properties of symmetric functions satisfying Assumption \ref{assumption}. 

\subsection{Properties of $F$}
We first show that such functions must be inverse concave. 

\begin{lemma}\label{lem:inverse-concave}
Let $F:\Gamma_+\to (0,\infty)$ be a smooth symmetric function. Assume that $F$ is homogeneous of degree one and strictly increasing in each variable, and satisfies the condition \eqref{equ-property3} of Assumption \ref{assumption}. Then $F$ is inverse concave.
\end{lemma}

\begin{proof}
Let $\mu_i=\kappa_i^{-1}$, so that $F_*(\mu)=1/{F(\kappa)}$. Write $\dot{F}^i=\frac{\partial F}{\partial \kappa_i}$ and
$\ddot{F}^{ij}=\frac{\partial^2 F}{\partial \kappa_i\partial \kappa_j}$. A direct computation gives
\[
\frac{\partial F_*}{\partial \mu_i}
=\frac{\dot{F}^i\kappa_i^2}{F^2},
\]
and
\[
\frac{\partial^2 F_*}{\partial \mu_i\partial \mu_j}
=
\frac{\kappa_i^2\kappa_j^2}{F^2}
\left(
2\frac{\dot{F}^i \dot{F}^j}{F}-\ddot{F}^{ij}-2\frac{\dot{F}^i}{\kappa_i}\delta_{ij}
\right).
\]
Therefore, for any $\xi=(\xi_1,\dots,\xi_n)\in \mathbb{R}^n$, setting $\eta_i=\kappa_i^2\xi_i$, we obtain
\begin{align*}
-\sum_{i,j=1}^n \frac{\partial^2 F_*}{\partial \mu_i\partial \mu_j}\,\xi_i\xi_j
=&\frac1{F^2}
\sum_{i,j=1}^n
\left(
\ddot{F}^{ij}+2\frac{\dot{F}^i}{\kappa_i}\delta_{ij}-2\frac{\dot{F}^i \dot{F}^j}{F}
\right)\eta_i\eta_j\\
=& I_1+I_2,
\end{align*}
where
\[
I_1=
\frac1{F^2}\sum_{i,j=1}^n
\left(
\ddot{F}^{ij}+\frac{\dot{F}^i}{\kappa_i}\delta_{ij}-\frac{\dot{F}^i \dot{F}^j}{F}
\right)\eta_i\eta_j
\]
and
\[
I_2=
\frac1{F^2}\left(\sum_{i=1}^n \frac{\dot{F}^i}{\kappa_i}\eta_i^2
-\frac1F\left(\sum_{i=1}^n \dot{F}^i\eta_i\right)^2\right).
\]

By the condition \eqref{equ-property3}, $I_1\ge 0$. On the other hand, since $\dot{F}^i>0$ and $F$ is homogeneous of degree one, Euler's identity gives $\sum_{i=1}^n \dot{F}^i\kappa_i=F$. Thus, by Cauchy--Schwarz,
\begin{align*}
\left(\sum_{i=1}^n \dot{F}^i\eta_i\right)^2
\leq &
\left(\sum_{i=1}^n \dot{F}^i\kappa_i\right)
\left(\sum_{i=1}^n \frac{\dot{F}^i}{\kappa_i}\eta_i^2\right)
=
F\sum_{i=1}^n \frac{\dot{F}^i}{\kappa_i}\eta_i^2.
\end{align*}
Therefore $I_2\ge 0$. Combining the above inequalities, we conclude that
\[
-\sum_{i,j=1}^n \frac{\partial^2 F_*}{\partial \mu_i\partial \mu_j}\,\xi_i\xi_j
\ge 0
\qquad\text{for all }\xi\in \mathbb{R}^n.
\]
Hence  $F_*$ is concave and so $F$ is inverse concave.
\end{proof}

We then show that the condition \eqref{equ-property3} is equivalent to the following log-convexity property. 
\begin{lemma}\label{lem:logF-convex}
Let $F:\Gamma_+\to (0,\infty)$ be a smooth symmetric function. Define
\begin{equation}\label{s3.Phi}
  \Phi(z_1,\dots,z_n):=\log F(e^{z_1},\dots,e^{z_n}),\qquad z=(z_1,\dots,z_n)\in\mathbb{R}^n.  
\end{equation}
Then the condition \eqref{equ-property3} is equivalent to that the function $\Phi$ is a convex function on $\mathbb{R}^n$.
\end{lemma}

\begin{proof}
Since $\kappa_i=e^{z_i}$, we have
\[\frac{\partial\Phi}{\partial z_i}
=\frac{\kappa_i\dot{F}^i}{F}, \]
and
\[\frac{\partial^2\Phi}{\partial z_i\partial z_j}
=\frac{\kappa_i\kappa_j\ddot{F}^{ij}}{F}
+\delta_{ij}\frac{\kappa_i\dot{F}^i}{F}
-\frac{\kappa_i\dot{F}^i\,\kappa_j\dot{F}^j}{F^2}.
\]
Hence, for any $w=(w_1,\dots,w_n)\in\mathbb{R}^n$, letting $y_i=\kappa_i w_i$, we obtain
\[\sum_{i,j}\frac{\partial^2\Phi}{\partial z_i\partial z_j}w_iw_j
=\frac{1}{F}\left(\sum_{i,j}\ddot{F}^{ij}y_iy_j
+\sum_i \frac{\dot{F}^i}{\kappa_i}y_i^2
-F^{-1}\Big(\sum_i \dot{F}^i y_i\Big)^2\right).
\]
Therefore $D^2\Phi\ge 0$ on $\mathbb{R}^n$ if and only if \eqref{equ-property3} holds for all $y\in\mathbb{R}^n$.
\end{proof}

The log-convexity property has the following consequence, which will be used to estimate the gradient terms in the proof of the curvature pinching estimate in Section \ref{sec4}.

\begin{corollary}\label{lem.cond-ii}
Let $F:\Gamma_+\to (0,\infty)$ be a smooth symmetric function. If $F$ satisfies the condition \eqref{equ-property3}, then 
\begin{equation}\label{eq:old-property2}
\big(\dot F^i\kappa_i-\dot F^j\kappa_j\big)(\kappa_i-\kappa_j)\ge 0
\qquad \text{for all } i\neq j.
\end{equation}
\end{corollary}
\proof 
Indeed, by Lemma \ref{lem:logF-convex} the condition \eqref{equ-property3} is equivalent to the convexity of the function $\Phi$ defined in \eqref{s3.Phi}. Since $F$ is symmetric, $\Phi$ is also symmetric. Therefore $\Phi$ is a symmetric convex function on $\mathbb{R}^n$, and hence (see \cite[Lemma 2]{EH89})
\[
(z_i-z_j)\big(\frac{\partial\Phi}{\partial z_i}-\frac{\partial\Phi}{\partial z_j}\big)\ge 0
\qquad \text{for all } i\neq j.
\]
Noting that
\[
\frac{\partial\Phi}{\partial z_i}=\frac{\kappa_i\dot F^i}{F},
\qquad \kappa_i=e^{z_i},
\]
and that $\kappa_i-\kappa_j$ has the same sign as $z_i-z_j$, we obtain
\[
(\kappa_i-\kappa_j)\Big(\frac{\kappa_i\dot F^i}{F}-\frac{\kappa_j\dot F^j}{F}\Big)\ge 0,
\]
which is exactly \eqref{eq:old-property2}.
\endproof

\subsection{Geometric-mean lower bound}
Using the log-convexity property in Lemma \ref{lem:logF-convex}, we also  prove a useful lower bound of $F$ in terms of the geometric mean, which will be used crucially in our proof of curvature pinching estimate in Section \ref{sec4}. 

\begin{proposition}\label{lem:gm-lower-bound}
 Let $F:\Gamma_+\to (0,\infty)$ be a smooth symmetric function on $\Gamma_+$ satisfying the Assumption \ref{assumption}. Then
\begin{equation}\label{equ-f-lower}
F(\kappa)\ge \left(\prod_{i=1}^n \kappa_i\right)^{1/n}
\qquad \text{for all }\kappa\in \Gamma_+.
\end{equation}
\end{proposition}

\begin{proof}
For each permutation $\sigma\in S_n$, where $S_n$ denotes the set of all permutations of $\{1,\dots,n\}$, let $P_\sigma z$ denote the corresponding permutation of $z$. Since $F$ is symmetric, the function $\Phi$ defined in \eqref{s3.Phi} is symmetric, and hence
\[\Phi(P_\sigma z)=\Phi(z)
\qquad \text{for all }\sigma\in S_n.
\]
On the other hand,
\[\frac1{n!}\sum_{\sigma\in S_n} P_\sigma z=(\bar z,\dots,\bar z),\quad \bar{z}=\frac{1}{n}\sum_{i=1}^nz_i.
\]
Therefore, by convexity of $\Phi$,
\[\Phi(\bar z,\dots,\bar z)
=\Phi\!\left(\frac1{n!}\sum_{\sigma\in S_n}P_\sigma z\right)
\le\frac1{n!}\sum_{\sigma\in S_n}\Phi(P_\sigma z)
=\Phi(z).
\]
Since $F$ is homogeneous of degree one and normalized by $F(1,\dots,1)=1$, exponentiating the above inequality yields
\begin{align*}
    F(e^{z_1},\dots,e^{z_n})\ge &F(e^{\bar z},\dots,e^{\bar z})\\
    =&e^{\bar z}F(1,\dots,1)    =e^{\bar z}.
\end{align*}
Since $e^{\bar z}=(e^{z_1}\cdots e^{z_n})^{1/n}$, the inequality is equivalent to \eqref{equ-f-lower}.
\end{proof}

\subsection{Examples}
Using Lemma \ref{lem:logF-convex}, it is easy to verify that both
\[
F(\kappa)=\Big(\frac1n\sum_{i=1}^n \kappa_i^r\Big)^{1/r},\quad r>0.
\]
and 
\[
F(\kappa)=E_k^{1/k}(\kappa),\quad k=1,\cdots,n
\]
satisfy Assumption \ref{assumption}.  Indeed, the condition (i) can be verified directly; to verify condition (ii), it suffices to verify $\Phi(z)=\log F(e^z)$ is convex. For the first function, 
\[\Phi(z)
=\frac1r\log\Big(\frac1n\sum_{i=1}^n e^{rz_i}\Big).
\]
This is the logarithm of a sum of exponentials of linear functions, multiplied by the positive constant $1/r$; hence it is convex. For the second function, 
\[\Phi(z)
=\frac1k\log\Bigg(
\binom{n}{k}^{-1}
\sum_{1\le i_1<\cdots<i_k\le n}
e^{z_{i_1}+\cdots+z_{i_k}}
\Bigg).
\]
Again, this is a log-sum-exp function, hence convex.

Generally, assume that $F^{(1)},\dots,F^{(m)}$ satisfy Assumption \ref{assumption}, and let $\theta_\alpha>0$ and $\sum_{\alpha=1}^m \theta_\alpha=1$.Then
\[
F(\kappa):=\prod_{\alpha=1}^m \big(F^{(\alpha)}(\kappa)\big)^{\theta_\alpha}
\]
also satisfies Assumption \ref{assumption}. This follows that 
\[\Phi(z):=\log F(e^z)
=\sum_{\alpha=1}^m \theta_\alpha
\log F^{(\alpha)}(e^z),
\]
and  each $z\mapsto \log F^{(\alpha)}(e^z)$ is symmetric and convex.

\section{Curvature pinching estimate}\label{sec4}

Let $F$ satisfy Assumption \ref{assumption}. If the initial hypersurface $M_0$ has positive sectional curvature, then an argument analogous to that in \cite[Theorem 3.1]{ACW21} shows that the evolving hypersurface $M_t$ of the flow \eqref{equ-f} in hyperbolic space $\mathbb{H}^{n+1}$ also has positive sectional curvature for $t>0$. 
\begin{lemma}\label{s3.lem-psc}
	Let $F$ satisfy Assumption \ref{assumption}. If the initial hypersurface $M_0$ has positive sectional curvature, then along the flow \eqref{equ-f} in $\mathbb{H}^{n+1}$ the evolving hypersurface $M_t$ has positive sectional curvature for $t>0$.
\end{lemma}

Indeed, this result follows from the argument as in the proof of \cite[Theorem 3.1]{ACW21} for the volume‑preserving curvature flow, after  setting the global term $\phi(t)$ there to zero. We therefore omit the details.

Next, we derive the curvature pinching estimate of the flow \eqref{equ-f} under Assumption \ref{assumption}. This is a key step in studying contracting curvature flows. 
\begin{proposition}\label{s4.lem-pinc}
	Let $F$ satisfy Assumption \ref{assumption}. If the initial hypersurface $M_0$ has positive sectional curvature, then there exists a constant $C$ depending on $n$ and $M_0$ such that along the flow \eqref{equ-f} in $\mathbb{H}^{n+1}\ (n\geq 3)$ the evolving hypersurface $M_t$ satisfies
	\begin{equation}
		\label{equ-pinching}
		\kappa_n \leq C \kappa_1
	\end{equation}
    for all $t>0$.
\end{proposition}
\begin{proof}
	The sectional curvature defines a smooth function on the Grassmannian bundle of two-dimensional subspaces of the tangent bundle $TM$. For convenience we lift it to a function on the orthonormal frame bundle $O(M)$ over $M$: Given a point $x\in M$, a time $t\geq 0$, and a frame $\mathbb{O}=\{e_1,e_2,\cdots,e_n\}$ for $T_x M$ which is orthonormal with respect to $g(x,t)$, we define
	\begin{equation*}
		G(x,t,\mathbb{O}) = h_{(x,t)}(e_1,e_1)h_{(x,t)}(e_2,e_2) - h_{(x,t)}(e_1,e_2)^2 - 1- \varepsilon F(x,t)^2,
	\end{equation*}
    where $\varepsilon\in(0,\varepsilon_0)$ is small enough such that $G>0$ at $t=0$. Here
    \begin{equation*}
    	\varepsilon_0 = \min\limits_{M_0}\dfrac{\kappa_1\kappa_2-1}{F^2}.
    \end{equation*}
    Consider a point $(x_0,t_0)$ and a frame $\mathbb{O}_0=\{\overline{e}_1,\overline{e}_2,\cdots,\overline{e}_n\}$ at which a new minimum of the function $G$ is attained, so that $$G(x,t,\mathbb{O})\geq G(x_0,t_0,\mathbb{O}_0)$$
    for all $x\in M$, all $t\in [0,t_0]$, and all $\mathbb{O}\in F(M)_{(x,t)}$. Since $\mathbb{O}_0$ achieves the minimum of $G$ over the fiber $F(M)_{(x_0,t_0)}$, the vectors $\overline{e}_1$ and $\overline{e}_2$ can be rotated to become eigenvectors of $h_{(x_0,t_0)}$ corresponding to $\kappa_1$ and $\kappa_2$, where $\kappa_1\leq\cdots\leq \kappa_n$ are the principal curvatures at $(x_0,t_0)$. Moreover, $G$ is invariant under rotation in the subspace orthogonal to $\overline{e}_1$ and $\overline{e}_2$. Consequently, we may assume without loss of generality that $h(\overline{e}_i,\overline{e}_i)=\kappa_i$ and $h(\overline{e}_i,\overline{e}_j)=0$ for $i\neq j$.
    
    We derive the differential inequality satisfied by $G$ at the minimal point at $(x_0,t_0,\mathbb{O}_0)$. Note that by the evolution equation \eqref{s2.evl-g} of the metric, the frame $\mathbb{O}(t)$ for $T_x M$ defined by
    \begin{equation*}
    	\dfrac{\mathrm{d}}{\mathrm{d}t}e_i(t) = F \mathcal{W}(e_i)
    \end{equation*}
    remains orthonormal with respect to $g(x,t)$ if $e_i(t_0)=\overline{e}_i$ for each $i$.  We first compute the 
    time derivative of $G$ at $(x_0,t_0,\mathbb{O}_0)$. Combining \eqref{s2.evlh2} and \eqref{equ-psi}, we have 
    \begin{align}
    	\left.\dfrac{\partial}{\partial t} G \right|_{(x_0,t_0,\mathbb{O}_0)} = & \kappa_1 \left(\frac{\partial}{\partial t}h_{22}+2h\left(\frac{\partial}{\partial t}e_2,e_2\right)\right) \nonumber\\
        &+ \kappa_2 \left(\frac{\partial}{\partial t}h_{11}+2h\left(\frac{\partial}{\partial t}e_1,e_1\right)\right) - 2 \varepsilon F \dfrac{\partial}{\partial t}F \notag\\
    	= & \kappa_1 \left(\frac{\partial}{\partial t}h_{22}+2F\kappa_2^2\right) \nonumber\\
        &+ \kappa_2 \left(\frac{\partial}{\partial t}h_{11}+2F\kappa_1^2\right) - 2 \varepsilon F \dfrac{\partial}{\partial t}F \notag\\
        = &  \kappa_1 \dot{F}^{k\ell} \nabla_k\nabla_{\ell} h_{22} + \kappa_2 \dot{F}^{k\ell}\nabla_k\nabla_{\ell} h_{11} \nonumber\\
        &+ \kappa_1\ddot{F}^{k\ell,pq}\nabla_2 h_{k\ell}\nabla_2 h_{pq} + \kappa_2 \ddot{F}^{k\ell,pq}\nabla_1 h_{k\ell}\nabla_1 h_{pq} \nonumber\\
    	& + 2 \left( \dot{F}^{k\ell}(h^2)_{k\ell} + \dot{F}^{k\ell}g_{k\ell} \right)\kappa_1\kappa_2 - 2 F (\kappa_1+\kappa_2) \nonumber\\
    	& - 2\varepsilon F \left( \dot{F}^{k\ell} \nabla_k \nabla_{\ell} F + \left( \dot{F}^{k\ell}(h^2)_{k\ell} - \dot{F}^{k\ell}g_{k\ell} \right)F  \right).
    	\label{equ-Gt-2}
    \end{align}
    
    The zero-order terms in \eqref{equ-Gt-2} satisfy
    \begin{align}
    	& 2 \left( \dot{F}^{k\ell}(h^2)_{k\ell} + \dot{F}^{k\ell}g_{k\ell} \right)\kappa_1\kappa_2 - 2 F (\kappa_1+\kappa_2) - 2 \varepsilon F^2 \left( \dot{F}^{k\ell}(h^2)_{k\ell} - \dot{F}^{k\ell}g_{k\ell} \right) \nonumber\\
    	=& 2 \sum\limits_{k=1}^n \dot{F}^k \kappa_k^2 G + 2 \sum\limits_{k=1}^n \dot{F}^k (\kappa_k-\kappa_1)(\kappa_k-\kappa_2) + 2 \varepsilon F^2 \dot{F}^{k\ell}g_{k\ell} ~
    	\geq ~ 0.   \label{equ-zero-2}
    \end{align}
    To estimate the Hessian terms in \eqref{equ-Gt-2}, we consider the second derivatives of $G$ along a curve on $O(M)$ defined as follows: We let $\gamma$ be any geodesic of $g(t_0)$ in $M$ with $\gamma(0)=x_0$, and define a frame $\mathbb{O}(s)=(e_1(s),e_2(s),\cdots,e_n(s))$ at $\gamma(s)$ by taking $e_i(0)=\overline{e}_i$ for each $i$, and
    \begin{equation*}
    	\nabla_s e_i(s) = \Gamma_{ij}e_j(s)
    \end{equation*}
    for some constant antisymmetric matrix $\Gamma$. Then we compute
    \begin{align}
    	\left.\dfrac{\mathrm{d}^2}{\mathrm{d}s^2}G(x(s),t_0,\mathbb{O}(s))\right|_{s=0} = & \kappa_2 \nabla_s^2 h_{11} + \kappa_1 \nabla_s^2 h_{22} + 2 \left( \nabla_s h_{22}\nabla_s h_{11} - (\nabla_s h_{12})^2 \right) \nonumber\\
    	& + 4 \sum\limits_{p>2} \Gamma_{1p}\kappa_2 \nabla_s h_{1p} + 4 \sum\limits_{p>2} \Gamma_{2p}\kappa_1 \nabla_s h_{2p} \nonumber\\
    	&+ 2 \sum\limits_{p>2} \Gamma_{1p}^2 \kappa_2(\kappa_p-\kappa_1) + 2 \sum\limits_{p>2} \Gamma_{2p}^2 \kappa_1 (\kappa_p-\kappa_2) \nonumber\\
    	& - 2 \varepsilon F \nabla_s^2 F - 2 \varepsilon (\nabla_s F)^2.  
    	\label{equ-d2G-2}
    \end{align}
    Since $G$ has a minimum at $(x_0,t_0,\mathbb{O}_0)$, the right hand side of \eqref{equ-d2G-2} is nonnegative for any choice of $\Gamma$. The optimal choice of $\Gamma$, which minimizes the expression \eqref{equ-d2G-2}, is given by
    \begin{equation*}
    	\Gamma_{kp}=\begin{cases}
    		-\dfrac{\nabla_s h_{kp}}{\kappa_p-\kappa_k},&k\in\{1,2\},p>2,\\
    		\dfrac{\nabla_s h_{kp}}{\kappa_k-\kappa_p},&p\in\{1,2\},k>2,\\
    		0,& \mbox{otherwise}.
    	\end{cases}
    \end{equation*}
    Thus, minimizing over $\Gamma$ in \eqref{equ-d2G-2} gives
    \begin{align}
    	0 \leq & \kappa_2 \nabla_s^2 h_{11} + \kappa_1 \nabla_s^2 h_{22} + 2 \left( \nabla_s h_{22}\nabla_s h_{11} - (\nabla_s h_{12})^2  \right) \nonumber\\
    	& - 2 \sum\limits_{p>2} \dfrac{\kappa_2}{\kappa_p-\kappa_1}(\nabla_s h_{1p})^2 - 2 \sum\limits_{p>2} \dfrac{\kappa_1}{\kappa_p - \kappa_2} (\nabla_s h_{2p})^2 \nonumber \\
    	& - 2 \varepsilon F \nabla_s^2 F - 2 \varepsilon (\nabla_s F)^2,
    	\label{equ-2nd-2}
    \end{align}
    with the terms on the second line regarded as vanishing if the denominators vanish (since the corresponding component of $\nabla h$ vanishes in that case). Using the formula \eqref{equ-secf}, we rewrite the gradient terms in \eqref{equ-Gt-2} as follows:
    \begin{align}
       \ddot{F}^{k\ell,pq}\nabla_2 h_{k\ell}\nabla_2 h_{pq}= &\sum\limits_{i,j} \ddot{F}^{ij}\nabla_2 h_{ii}\nabla_2 h_{jj} + 2\sum\limits_{i< j} \dfrac{\dot{F}^i -\dot{F}^{j}}{\kappa_i-\kappa_{j}}(\nabla_2 h_{ij})^2, \\
       \ddot{F}^{k\ell,pq}\nabla_1 h_{k\ell}\nabla_1 h_{pq} =&\sum\limits_{i,j} \ddot{F}^{ij}\nabla_1 h_{ii}\nabla_1 h_{jj} + 2\sum\limits_{i< j} \dfrac{\dot{F}^i -\dot{F}^{j}}{\kappa_i-\kappa_{j}} (\nabla_1 h_{ij})^2. \label{equ-HessF}
    \end{align}
    
    Substituting \eqref{equ-zero-2} and \eqref{equ-2nd-2} -- \eqref{equ-HessF} into \eqref{equ-Gt-2}, we obtain
    \begin{align}
    	\dfrac{\partial}{\partial t}G \bigg|_{(x_0,t_0,\mathbb{O}_0)} \geq &  \kappa_2 \left( \sum\limits_{i,j} \ddot{F}^{ij}\nabla_1 h_{ii}\nabla_1 h_{jj} + 2\sum\limits_{i< j} \dfrac{\dot{F}^i -\dot{F}^{j}}{\kappa_i-\kappa_{j}} (\nabla_1 h_{ij})^2 \right) \notag\\
    	& + \kappa_1 \left( \sum\limits_{i,j} \ddot{F}^{ij}\nabla_2 h_{ii}\nabla_2 h_{jj} + 2\sum\limits_{i< j} \dfrac{\dot{F}^i -\dot{F}^{j}}{\kappa_i-\kappa_{j}}(\nabla_2 h_{ij})^2 \right) \notag\\
    	& - 2 \sum\limits_{i=1}^n \dot{F}^i \left( \nabla_i h_{22}\nabla_i h_{11} - (\nabla_i h_{12})^2  \right) \nonumber\\
    	& + 2 \sum\limits_{i=1}^n \dot{F}^i \left( \sum\limits_{p>2} \dfrac{\kappa_2}{\kappa_p-\kappa_1}(\nabla_i h_{1p})^2 + \sum\limits_{p>2} \dfrac{\kappa_1}{\kappa_p-\kappa_2} (\nabla_i h_{2p})^2  \right) \nonumber \\
    	& + 2 \varepsilon \sum\limits_{i=1}^n \dot{F}^i (\nabla_i F)^2. \label{equ-ptg22}
    \end{align}

  Using \eqref{equ-property3} in Assumption \ref{assumption}, we estimate the terms involving second order derivatives of $F$ as follows 
    \begin{align}
        \sum\limits_{i,j} \ddot{F}^{ij}\nabla_1 h_{ii}\nabla_1 h_{jj}\geq &\dfrac{(\nabla_1 F)^2}{F} - \sum\limits_{i=1}^n \dfrac{\dot{F}^i}{\kappa_i}(\nabla_1 h_{ii})^2,\label{s3.d2F}\\
          \sum\limits_{i,j} \ddot{F}^{ij}\nabla_2 h_{ii}\nabla_2 h_{jj}\geq &\dfrac{(\nabla_2 F)^2}{F} - \sum\limits_{i=1}^n \dfrac{\dot{F}^i}{\kappa_i}(\nabla_2 h_{ii})^2. \label{s3.d2Fb}
    \end{align}
Under the Assumption \ref{assumption}, by Corollary \ref{lem.cond-ii} we have
\[
 \big(\dot F^i\kappa_i-\dot F^j\kappa_j\big)(\kappa_i-\kappa_j)\ge 0
\qquad \text{for all } i\neq j.   
\]
This implies that for any $i\neq j$,
    \begin{equation*}
    	\dfrac{\dot{F}^i-\dot{F}^j}{\kappa_i - \kappa_j} + \dfrac{\dot{F}^i}{\kappa_j}=\dfrac{\dot{F}^i\kappa_i-\dot{F}^j\kappa_j}{(\kappa_i-\kappa_j)\kappa_j} \geq 0.
    \end{equation*}
    Then we estimate
    \begin{align}
    	2 \sum\limits_{i<j} \dfrac{\dot{F}^i - \dot{F}^j}{\kappa_i-\kappa_j}(\nabla_1 h_{ij})^2
    	=& 2\dfrac{\dot{F}^1-\dot{F}^2}{\kappa_1-\kappa_2}(\nabla_2 h_{11})^2 + 2 \sum\limits_{\substack{1\leq i<j\leq n\\j>2}}\dfrac{\dot{F}^i -\dot{F}^j}{\kappa_i -\kappa_j}(\nabla_1 h_{ij})^2 \notag\\
    	\geq& -\left(\dfrac{\dot{F}^1}{\kappa_2}+\dfrac{\dot{F}^2}{\kappa_1}\right)(\nabla_2 h_{11})^2 - 2 \sum\limits_{i=1}^n\sum\limits_{\substack{j>i\\j>2}}\dfrac{\dot{F}^i}{\kappa_j}(\nabla_1 h_{ij})^2.
    	\label{equ-estd1}
    \end{align}
    Similarly,
    \begin{align}
    	2 \sum\limits_{i<j} \dfrac{\dot{F}^i - \dot{F}^j}{\kappa_i-\kappa_j}(\nabla_2 h_{ij})^2 	=& 2\dfrac{\dot{F}^1-\dot{F}^2}{\kappa_1-\kappa_2}(\nabla_1 h_{22})^2 + 2 \sum\limits_{\substack{1\leq i<j\leq n\\j>2}}\dfrac{\dot{F}^i -\dot{F}^j}{\kappa_i -\kappa_j}(\nabla_2 h_{ij})^2 \notag\\
    	\geq& -\left(\dfrac{\dot{F}^1}{\kappa_2}+\dfrac{\dot{F}^2}{\kappa_1}\right)(\nabla_1 h_{22})^2 - 2 \sum\limits_{i=1}^n\sum\limits_{\substack{j>i\\j>2}}\dfrac{\dot{F}^i}{\kappa_j}(\nabla_2 h_{ij})^2.
    	\label{equ-estd2}
    \end{align}
Substituting \eqref{s3.d2F} -- \eqref{equ-estd2} into \eqref{equ-ptg22}, we obtain
    \begin{align}
    	\left.\dfrac{\partial}{\partial t}G\right|_{(x_0,t_0,\mathbb{O}_0)} \geq & \kappa_2 \left( \dfrac{(\nabla_1 F)^2}{F} - \sum\limits_{i=1}^n \dfrac{\dot{F}^i}{\kappa_i}(\nabla_1 h_{ii})^2 \right)\nonumber\\
        &-\kappa_2\left(\left(\dfrac{\dot{F}^1}{\kappa_2}+\dfrac{\dot{F}^2}{\kappa_1}\right)(\nabla_2 h_{11})^2 + 2 \sum\limits_{i=1}^n\sum\limits_{\substack{j>i\\j>2}}\dfrac{\dot{F}^i}{\kappa_j}(\nabla_1 h_{ij})^2  \right) \nonumber\\
    	& +  \kappa_1 \left( \dfrac{(\nabla_2 F)^2}{F} - \sum\limits_{i=1}^n \dfrac{\dot{F}^i}{\kappa_{i}}(\nabla_2 h_{ii})^2\right)\nonumber\\
        &-\kappa_1\left( \left(\dfrac{\dot{F}^1}{\kappa_2}+\dfrac{\dot{F}^2}{\kappa_1}\right)(\nabla_1 h_{22})^2 + 2 \sum\limits_{i=1}^n\sum\limits_{\substack{j>i\\j>2}}\dfrac{\dot{F}^i}{\kappa_j}(\nabla_2 h_{ij})^2 \right) \notag\\
    	& + 2\sum\limits_{i=1}^n \dot{F}^i \left( -\nabla_i h_{22}\nabla_i h_{11} + (\nabla_i h_{12})^2 \right) \notag\\
    	& + 2 \sum\limits_{i=1}^n \dot{F}^{i}\left( \sum\limits_{j>2}\dfrac{\kappa_2}{\kappa_j-\kappa_1}(\nabla_i h_{1j})^2 +\sum\limits_{j>2}\dfrac{\kappa_1}{\kappa_j-\kappa_2}(\nabla_i h_{2j})^2 \right) \notag\\
    	& + 2\varepsilon \sum\limits_{i=1}^n \dot{F}^i (\nabla_i F)^2. \notag\\
    	\geq & \kappa_2\dfrac{(\nabla_1 F)^2}{F} + \kappa_1 \dfrac{(\nabla_2 F)^2}{F} \nonumber\\
        &- \sum\limits_{i=1}^2 \dfrac{\kappa_2}{\kappa_1}\dot{F}^i (\nabla_i h_{11})^2 - \sum\limits_{i=1}^2 \dfrac{\kappa_1}{\kappa_2}\dot{F}^i (\nabla_i h_{22})^2 \nonumber\\
    	& - 2 \sum\limits_{i=1}^n \dot{F}^i \nabla_i h_{22}\nabla_i h_{11} + 2\varepsilon \sum\limits_{i=1}^n \dot{F}^i (\nabla_i F)^2 \nonumber\\
    	= & \kappa_2 \dfrac{(\nabla_1 F)^2}{F} + \kappa_1 \dfrac{(\nabla_2 F)^2}{F} - \sum\limits_{i=1}^2 \dfrac{\dot{F}^i}{\kappa_1 \kappa_2}(\kappa_2 \nabla_i h_{11}+ \kappa_1 \nabla_i h_{22})^2 \nonumber\\
    	& - 2 \sum\limits_{i=3}^n \dot{F}^i \nabla_i h_{22}\nabla_i h_{11} + 2\varepsilon \sum\limits_{i=1}^n \dot{F}^i (\nabla_i F)^2.  \label{equ-2ndgauss2}
    \end{align}
    
    Since $(x_0,\mathbb{O}_0)$ is a minimum point of $G$ at time $t_0$, we have $\nabla_i G=0$ for $i=1,2,\cdots,n$, which yields the critical condition
    \begin{equation}
    	\kappa_2 \nabla_i h_{11} + \kappa_1 \nabla_i h_{22} =2\varepsilon F\nabla_i F,\ \ i=1,\cdots,n.
    	\label{equ-critical-2}
    \end{equation} 
	The fact that $G\geq 0$ at $(x_0,\mathbb{O}_0)$ implies 
    \begin{equation}\label{s4.Ggeq0}
        \kappa_1\kappa_2 \geq 1+\varepsilon F^2\geq \varepsilon F^2
    \end{equation}
holds at the minimum point $x_0$.    Using \eqref{equ-critical-2} and \eqref{s4.Ggeq0},  we estimate the terms in \eqref{equ-2ndgauss2} involving $i=1,2$ as follows:
    \begin{align}
    	&\kappa_2 \dfrac{(\nabla_1 F)^2}{F} + \kappa_1 \dfrac{(\nabla_2 F)^2}{F} - \sum\limits_{i=1}^2 \dfrac{\dot{F}^i}{\kappa_1 \kappa_2}(\kappa_2 \nabla_i h_{11}+ \kappa_1 \nabla_i h_{22})^2 + 2\varepsilon \sum\limits_{i=1}^2 \dot{F}^i (\nabla_i F)^2 \nonumber\\
    	\geq & \sum\limits_{i=1}^2 \dfrac{\varepsilon F^2}{\kappa_i}\dfrac{(\nabla_i F)^2}{F} - \dfrac{4\varepsilon^2 F^2}{\kappa_1 \kappa_2} \sum\limits_{i=1}^2 \dot{F}^i (\nabla_i F)^2 +  2\varepsilon \sum\limits_{i=1}^2 \dot{F}^i (\nabla_i F)^2 \nonumber\\
    	\geq & \sum\limits_{i=1}^2 \dfrac{\varepsilon F^2}{\kappa_i}\dfrac{(\nabla_i F)^2}{F} - 2\varepsilon \sum\limits_{i=1}^2 \dot{F}^i (\nabla_i F)^2 \notag\\
    	= & \sum\limits_{i=1}^2 \frac{\varepsilon}{\kappa_i}\left(F-2\dot{F}^{i}\kappa_i\right)(\nabla_i F)^2 \geq 0,
    	\label{equ-est1}
    \end{align}
    where in the last line we used the inequality
    \begin{equation*}
    	F = \sum\limits_{i=1}^n \dot{F}^i \kappa_i \geq 2 \dot{F}^j\kappa_j,\ \ j=1,2,
    \end{equation*}
	which follows from \eqref{eq:old-property2} in Corollary \ref{lem.cond-ii} and $n\geq 3$. For the remaining terms in \eqref{equ-2ndgauss2} with $i\geq 3$, we again employ the critical condition \eqref{equ-critical-2} and \eqref{s4.Ggeq0}.  For each such $i$ we have
    \begin{align*}
    	& -\nabla_i h_{22}\nabla_i h_{11} + \varepsilon (\nabla_i F)^2 \nonumber\\
    	=& \dfrac{\kappa_2}{\kappa_1}(\nabla_i h_{11})^2 - 2\varepsilon\dfrac{F}{\kappa_1}\nabla_i F \nabla_i h_{11} + \varepsilon(\nabla_i F)^2 \notag\\
    	\geq & \dfrac{\varepsilon F^2}{\kappa_1^2}(\nabla_i h_{11})^2 - 2\varepsilon\dfrac{F}{\kappa_1}\nabla_i F \nabla_i h_{11} + \varepsilon(\nabla_i F)^2\notag\\
    	=& \varepsilon \left( \dfrac{F}{\kappa_1}\nabla_i h_{11} - \nabla_i F \right)^2\notag~
    	\geq  0.
    \end{align*}
   Consequently,
    \begin{equation}
    	-2\sum\limits_{i=3}^n \dot{F}^i \nabla_i h_{22}\nabla_i h_{11} + 2\varepsilon \sum\limits_{i=3}^n \dot{F}^i (\nabla_i F)^2 \geq 0.
    	\label{equ-est2}
    \end{equation}

    Substituting \eqref{equ-est1} and \eqref{equ-est2} into \eqref{equ-2ndgauss2}, we conclude
    \begin{equation*}
    	\left.\frac{\partial}{\partial t}G\right|_{(x_0,t_0,\mathbb{O}_0)}\geq 0.
    \end{equation*}
    By the maximum principle, the minimum of $G$ cannot decrease. Since initially $G>0$, it remains positive for all $t>0$. Then using the property \eqref{equ-f-lower} in Proposition \ref{lem:gm-lower-bound} for the function $F$,  $M_t$ satisfies
    \begin{equation*}
    	\kappa_1 \kappa_2 >1+ \varepsilon F^2 \geq \varepsilon \left(\prod_{i=1}^n \kappa_i\right)^{\frac{2}{n}} \geq \varepsilon \kappa_1^{\frac{n-2}{n}} \kappa_2 \kappa_n^{\frac{2}{n}}.
    \end{equation*}
  From this, we obtain the desired pinching estimate 
    \[\kappa_n\leq C\kappa_1\] 
    with $C=\varepsilon^{-n/2}$, which depends only on $n$ and $M_0$.
\end{proof}

\section{Proof of Theorem \ref{thm-main}}\label{sec5}
Let $F$ satisfy Assumption \ref{assumption} and denote by $[0,T)$ the maximal existence interval of the flow \eqref{equ-f}. It follows from \cite[Proposition 1.1]{AC17} that $T<\infty$. In this section, we prove that the solution $M_t$ of the flow \eqref{equ-f} contracts to a point as $t\to T$, and a rescaled solution converges smoothly to a geodesic sphere. 

\subsection{Contraction to a point}

We first prove that the solution $M_t$ of \eqref{equ-f} remains smooth and contracts to a point as $t\to T$.

\begin{proposition} \label{lem-point}
Let $F$ satisfy Assumption \ref{assumption}. If the initial hypersurface $M_0$ has positive sectional curvature, then along the flow \eqref{equ-f} in hyperbolic space $\mathbb{H}^{n+1}$, the evolving hypersurfaces $M_t$ remain smooth for all $t\in[0,T)$ and contract to a point as $t\to T$.
\end{proposition}

Firstly, we prove the lower bound of the speed function $F$ and the smallest principal curvature $\kappa_1$ along the flow \eqref{equ-f}.
\begin{lemma}
	Let $M_t$, $t\in[0,T)$ be a smooth solution to the flow \eqref{equ-f} in $\mathbb{H}^{n+1}$. If $M_0$ has positive sectional curvature, then we have
	\begin{equation}
		\label{equ-psilbd}
		F\geq \min\limits_{M_0}F
	\end{equation}
    and
    \begin{equation}
    	\label{equ-k1low}
    	\kappa_i \geq C(n,M_0),\quad \forall~i=1,\cdots,n
    \end{equation}
    on $M_t$ for $t\in[0,T)$.
\end{lemma}
\begin{proof}
To prove the lower bound \eqref{equ-psilbd}, we apply maximum principle to the evolution equation 
\begin{equation}\label{s5.equF}
\partial_t F =  \dot{F}^{k\ell}\nabla_k\nabla_{\ell}F + \left( \dot{F}^{ij}(h^2)_{ij}- \dot{F}^{ij}g_{ij} \right)F
\end{equation}
of $F$ (see \eqref{equ-psi}). We need to estimate the sign of the zero order terms in \eqref{s5.equF}. In a local orthonormal frame diagonalizing $h_{ij}$, one has
\[\dot F^{ij}(h^2)_{ij}-\dot F^{ij}g_{ij}
=\sum_{i=1}^n \dot F^i\kappa_i^2-\sum_{i=1}^n \dot F^i.
\]
Without loss of generality, we may assume that $\kappa_1\le \kappa_2\le \cdots \le \kappa_n$. By Lemma \ref{s3.lem-psc}, the evolving hypersurface $M_t$ has positive sectional curvature for $t\in [0,T)$. Then $\kappa_i\kappa_j>1$ for all $i\neq j$. In particular, we have 
\[
\kappa_1\kappa_i>1 \qquad \text{for all } i=2,\dots,n.
\]
Therefore,
\begin{align*}
		\sum\limits_{i=1}^n \dot{F}^i \kappa_i^2 - \sum\limits_{i=1}^n \dot{F}^i
		\geq &  \sum\limits_{i=1}^n \dot{F}^i\kappa_i^2 - \sum\limits_{i=2}^n \dot{F}^i\kappa_i\kappa_1 - \dot{F}^1\kappa_1\kappa_2 \notag\\
		= & \sum\limits_{i=3}^n \dot{F}^i \kappa_i(\kappa_i-\kappa_1) + \left( \dot{F}^1\kappa_1 - \dot{F}^2\kappa_2 \right)(\kappa_1-\kappa_2).
	\end{align*}
The first term on the right-hand side is nonnegative, since $\dot F^i>0$ and $\kappa_i\ge \kappa_1$ for $i\ge 3$. The second term is also nonnegative, by applying \eqref{eq:old-property2} in Corollary \ref{lem.cond-ii} with $(i,j)=(1,2)$. Therefore, 
\[\dot F^{ij}(h^2)_{ij}-\dot F^{ij}g_{ij}\geq 0,\]
and the maximum principle applied to the evolution equation \eqref{s5.equF} implies the lower bound \eqref{equ-psilbd} of $F$.
    
Since $F$ is monotone increasing with respect to each argument and is normalized such that $F(1,\cdots,1)=1$, we have 
    \[F=F(\kappa_1,\cdots,\kappa_n)\leq \kappa_nF(1,\cdots,1)=\kappa_n.\] 
    This together with the pinching estimate \eqref{equ-pinching} and \eqref{equ-psilbd} implies that
    \begin{equation*}
    	\kappa_1 \geq C^{-1}\kappa_n \geq C^{-1}F \geq C^{-1}\min\limits_{M_0}F =: C(n,M_0).
    \end{equation*}
\end{proof}

Let $\rho_+(t)$ and $\rho_-(t)$ be the outer radius and inner radius of the domain $\Omega_t$ enclosed by $M_t$ respectively, defined by
\begin{align*}
	\rho_+(t) &= \inf \left\lbrace \rho:\Omega_t\subset B_{\rho}(p)\ \mbox{for some}\ p\in\mathbb{H}^{n+1} \right\rbrace,\\
	\rho_-(t) & = \sup\left\{ \rho:B_{\rho}(p) \subset \Omega_t\ \mbox{for some}\ p\in\mathbb{H}^{n+1} \right\}. 
\end{align*}
By the pinching estimate for the principal curvatures \eqref{equ-pinching}, we can project the evolving domains $\Omega_t$ enclosed by $M_t$ onto domains $\hat{\Omega}_t$ in the Euclidean unit ball $B_1(0)$ via the Klein model of hyperbolic space. This projection preserves convexity, and hence each $\hat{\Omega}_t$ remains strictly convex. Moreover, there is an explicit relation between the principal curvatures $\kappa_i$ of $M_t$ and the principal curvatures $\hat{\kappa}_i$ of the corresponding hypersurface $\partial\hat{\Omega}_t$ in Euclidean space (see \cite[\S 5]{AW18}). Therefore, the pinching estimate \eqref{equ-pinching} yields a corresponding pinching estimate for $\hat{\kappa}_i$. Applying a result of Andrews \cite{And94} for pinched hypersurfaces in Euclidean space, we then obtain that the outer radius of $\hat{\Omega}_t$ is bounded above by a constant multiple of its inner radius. This in turn implies the following estimate relating the outer and inner radius of $\Omega_t$ (see \cite[\S 6]{Ger15} or \cite[\S 6]{Yu16}):
\begin{equation}
	\rho_+(t) \leq C\rho_-(t),\ \ \mbox{for}\ t\in[t_0,T),
	\label{equ-rhopin}
\end{equation}
where $t_0$ is sufficiently close to $T$.

Next we use the technique of Tso \cite{Tso85} to show that $F$ remains bounded as long as the flow \eqref{equ-f} encloses a non-vanishing volume. Assume that there exists a geodesic ball $B_{\rho}(x_0)\subset\Omega_t$ for $t\in[0,t_1]$, where $t_1\in [t_0,T)$. Since $M_t$ is strictly convex by \eqref{equ-k1low}, we can write $M_t=\mathrm{graph}~u(\cdot,t)$ as graphs in polar coordinates centered at $x_0$. Then $u\geq \rho$ for all $t\in [0,t_1]$. By the comparison principle, the latter hypersurface is contained in the earlier one, then we have an upper bound on $u\leq 2\rho_+(0)$ only depending on $M_0$. Denote by $\partial_r$ the gradient vector at $x\in M_t$ along the geodesic from $x_0$ to $x$. The support function of $M_t$ with respect to $x_0$ is defined by $\chi(x,t)= \sinh u(x,t)\langle \partial_r,\nu \rangle$. Due to the strict convexity of $M_t$ and $\rho\leq u \leq 2\rho_+(0)$, we have 
\begin{equation}
	\sinh\rho \leq \chi \leq\sinh(2\rho_+(0))\ \ \mbox{for\ all}\ t\in [0,t_1].
	\label{equ-supplbd}
\end{equation}
\begin{lemma}\label{s4.lem-F}
	Let $M_t$ be a solution to the flow \eqref{equ-f} whose initial hypersurface $M_0$ has positive sectional curvature. Let $t_1$ and $\rho$ be as above. Then there exists a constant $C$ depending only on $n$, $\rho$ and $M_0$ such that
	\begin{equation}
		F \leq C
		\label{equ-psiupp}
	\end{equation}
    on $M_t$ for $t\in [0,t_1]$. 
\end{lemma}
\begin{proof}
	Define a function
	\begin{equation*}
		\varphi = \dfrac{F}{\chi-\frac{1}{2}\sinh\rho},
	\end{equation*}
    which is well defined on $M_t$ for all $t\in[0,t_1]$. Recall that the support function $\chi$ satisfies (see \cite[Section 4]{AW18})
    \begin{equation}
    	\dfrac{\partial}{\partial t}\chi = \dot{F}^{k\ell}\nabla_k\nabla_{\ell}\chi + \chi\dot{F}^{k\ell}(h^2)_{k\ell}- 2\cosh u F.\label{equ-supp}
    \end{equation}
    Combining evolution equations \eqref{equ-psi} and \eqref{equ-supp}, we derive that $\varphi$ satisfies the evolution equation
    \begin{align*}
    	\dfrac{\partial}{\partial t}\varphi =& \dot{F}^{k\ell}\nabla_k\nabla_{\ell}\varphi + \dfrac{2}{\chi-\frac{1}{2}\sinh\rho}\dot{F}^{k\ell}\nabla_k\varphi\nabla_{\ell}\chi \notag\\
    	& + \left( 2\cosh u - \dfrac{1}{2}\sinh\rho\dfrac{\dot{F}^{k\ell}(h^2)_{k\ell}}{F} \right)\varphi^2 - \dot{F}^{k\ell}g_{k\ell}\varphi.
    \end{align*}
    Since $F$ is inverse concave (Lemma \ref{lem:inverse-concave}), using the inequality \eqref{equ-inv2} and the lower bound \eqref{equ-supplbd} of $\chi$, we have 
    \begin{equation*}
    	\dfrac{\dot{F}^{k\ell}(h^2)_{k\ell}}{F} \geq F = \varphi \cdot \left( \chi - \dfrac{1}{2}\sinh\rho \right) \geq \dfrac{1}{2}\varphi\sinh\rho.
    \end{equation*}
  We also have $\dot{F}^{k\ell}g_{k\ell}\geq 0$. Then 
    \begin{align}
    	\frac{\partial}{\partial t}\varphi 
    	\leq & \dot{F}^{k\ell}\nabla_k\nabla_{\ell}\varphi + \dfrac{2}{\chi-\frac{1}{2}\sinh\rho}\dot{F}^{k\ell}\nabla_k\varphi\nabla_{\ell}\chi \nonumber\\
    	& + \left( 2\cosh u - \left(\dfrac{1}{2}\sinh\rho \right)^2\varphi \right)\varphi^2.
    	\label{equ-phi}
    \end{align}
    
Set
\[
A:=2\cosh\bigl(2\rho_+(0)\bigr)\Bigl(\frac12\sinh\rho\Bigr)^{-2}.
\]
Since $u\le 2\rho_+(0)$ on $[0,t_1]$, we have $\cosh u\le \cosh(2\rho_+(0))$. We claim that
\[
\varphi\le \max\left\{\max_{M_0}\varphi,\;A\right\}
\qquad\text{on } M_t,\quad t\in[0,t_1].
\]
Indeed, otherwise there would exist a first time $t_*\in(0,t_1]$ and a point $x_*\in M_{t_*}$ such that $\varphi(x_*,t_*)>\max\{\max_{M_0}\varphi,A\}$ and $\varphi(\cdot,t_*)$ attains its spatial maximum at $x_*$. At $(x_*,t_*)$, we have
\[
\nabla\varphi=0,
\qquad
\dot F^{k\ell}\nabla_k\nabla_\ell\varphi\le 0,
\]
and thus by \eqref{equ-phi},
\[
0\le \frac{\partial}{\partial t}\varphi
\le \left(2\cosh u-\Bigl(\frac12\sinh\rho\Bigr)^2\varphi\right)\varphi^2
<0,
\]
a contradiction. Therefore,
\[
\varphi\le \max\left\{\max_{M_0}\varphi,\;A\right\}
\quad\text{on } M_t,\ t\in[0,t_1].
\]

Combining this with the upper bound $\chi\le \sinh(2\rho_+(0))$ in \eqref{equ-supplbd}, we obtain
\[
F=\varphi\Bigl(\chi-\frac12\sinh\rho\Bigr)
\le \varphi\,\chi
\le \sinh(2\rho_+(0))
\max\left\{\max_{M_0}\varphi,\;A\right\},
\]
which proves \eqref{equ-psiupp}.
\end{proof}

We now complete the proof of Proposition \ref{lem-point}.
\begin{proof}[Proof of Proposition \ref{lem-point}]
The preceeding lemma implies that a smooth solution of the flow \eqref{equ-f} exists as long as the evolving domain encloses a non-zero volume. In fact, observing that $\kappa_1\leq F$, the upper bound of $F$ implies an upper bound on $\kappa_1$. This together with the pinching estimate \eqref{equ-pinching} and the lower bound \eqref{equ-k1low} for $\kappa_1$ implies  two-sided positive bounds for all the principal curvatures of $M_t$ for $t\in [0,t_1]$. Consequently, the coefficients $\dot{F}^{ij}$ appearing in the second-order part of the evolution equation \eqref{equ-psi} have eigenvalues bounded above and below by positive constants, so the flow \eqref{equ-f} remains uniformly parabolic. 
	Since the functions $F$ we considered are inverse concave (Lemma \ref{lem:inverse-concave}), we can apply an argument similar to that in \cite[Section 5]{AW18} to derive higher regularity estimates. Hence the solution can be extended beyond time $t_1$.  Therefore the inner radius $\rho_-(t)\to 0$ as $t$ approaches the maximal existence time $T$. By estimate \eqref{equ-rhopin}, the outer radius $\rho_+(t)$ also tends to 0 as $t\to T$. In summary, the flow \eqref{equ-f} remains smooth and contracts to a point as $t\to T$.
\end{proof}

\subsection{Roundness estimate}\label{sec6}
To study the asymptotic behavior of the flow as $t\to T$, we consider a rescaled solution of \eqref{equ-f} centered at the final point, and show that the rescaled hypersurfaces converge to a geodesic sphere. The main idea is to derive a roundness estimate by means of the Stampacchia iteration, following Huisken \cite{Hui86} for the mean curvature flow in Euclidean space, Gerhardt \cite{Ger15} for contracting flows in the sphere, and Yu \cite{Yu16} for flows of $h$-convex hypersurfaces in hyperbolic space.

\begin{proposition}
	\label{prop-roundness}
	Let $F$ satisfy Assumption \ref{assumption}. If, in addition, $F$ is strictly concave or $F=\frac{1}{n}H$, then there exist constants $\delta>0$ and $c_0>0$, depending only on $n$ and $M_0$, such that
	\begin{equation}
		|A|^2 - nF^2 \leq c_0 F^{2-\delta}
		\label{equ-roundness}
	\end{equation}
    on $M_t$ for all $t\in [0,T)$.
\end{proposition}

For the mean curvature flow case $F=\frac{1}{n}H$, the argument can be found in Huisken \cite{Hui86,Hui8602,Hui87}. When $F$ is strictly concave, once the curvature pinching estimate \eqref{equ-pinching} has been established, the roundness estimate \eqref{equ-roundness} can be derived by an argument similar to those in Gerhardt \cite{Ger15} and Yu \cite{Yu16}. The key quantity is
\begin{equation*}
	\psi_{\sigma}=\dfrac{|A|^2-nF^2}{F^{2-\sigma}},
\end{equation*}
where $\sigma\in(0,1)$, and one estimates the $L^p$-norm of $\psi_\sigma$ for sufficiently large $p$. The curvature pinching estimate \eqref{equ-pinching} is used to control certain quadratic terms by $\psi_\sigma$, while the strict concavity of $F$ allows one to extract favorable negative terms, namely (see \cite[Lemma 7.12]{And94})
\[
\ddot{F}^{k\ell,rs}\nabla_i h_{k\ell}\nabla_j h_{rs} \leq -c |A|^{-1}|\nabla A|^2g_{ij}.
\]
We refer the reader to \cite{Ger15,Yu16} for the details.

\subsection{Convergence of the rescaled flow}

With the roundness estimate in hand, we can rescale the flow with respect to the final point and show that it converges to a geodesic sphere exponentially. 

If the initial hypersurface is a geodesic sphere in $\mathbb{H}^{n+1}$, then the evolving hypersurfaces under the flow \eqref{equ-f} remain geodesic spheres with the same center and with radius $\Theta(t,T)$ satisfying
\[
\frac{\mathrm{d}}{\mathrm{d}t}\Theta(t,T) = -\coth\Theta.
\]
Such a spherical solution contracts to a point in finite time. We choose the initial sphere so that its maximal existence time coincides with the maximal existence time $T$ in Proposition \ref{lem-point}.

For a general solution, the comparison principle implies that $\Theta(t,T)$ controls the size of $M_t$. Writing $M_t$ as a radial graph $u(x,t)$ in geodesic polar coordinates centered at the final point, we have
\[
\inf_{M_t} u(\cdot,t) \le \Theta(t,T) \le \sup_{M_t} u(\cdot,t).
\]
To study the limit shape, we introduce the rescaled time
\begin{equation*}
	\tau = -\log\Theta(t,T),\qquad \tau\in[0,\infty),
\end{equation*}
and the rescaled quantities $\tilde{F}= \Theta F$ and $\tilde{\kappa}_i = \Theta\kappa_i.$ Since $\Theta(t,T)\to 0$ as $t\to T$,  $\tau$ increases from $0$ to $\infty$. Because of the pinching estimate \eqref{equ-rhopin}, we may restrict ourselves to the interval $[t_0,T)$ where $\Theta$ is already small; set $\tau_0 = -\log\Theta(t_0,T)$.

The rescaled function $\tilde{u}(x,\tau)=u/\Theta$ satisfies a parabolic equation:
\begin{align*}
	\dfrac{\partial}{\partial\tau}\tilde{u} =& -\tanh(\Theta)vF+ \tilde{u} \notag\\
	=& -\Theta^{-1}\tanh(\Theta)v \tilde{F} + \tilde{u},
\end{align*}
with $v^2 = 1 + {|\overline{\nabla}u|_{g_{\mathbb{S}^n}}^2}/{\sinh^2 u}.$ Using the roundness estimate \eqref{equ-roundness} and a similar argument as in \cite[\S 7-8]{Ger15} and \cite[\S 7-8]{Yu16}, we can establish uniform $C^\infty$ estimates for $\tilde{u}(x,\tau)$ on $(x,\tau)\in \mathbb{S}^n \times [\tau_0,\infty)$, and show that $\tilde{u}$ converges to $1$ as $\tau\to\infty$ exponentially in $C^\infty(\mathbb{S}^n)$. We omit the details. 

This completes the proof of Theorem \ref{thm-main}.

\section*{Acknowledgements}
The research was supported by National Key Research and Development Program of China 2021YFA1001800, National Natural Science Foundation of China No. 12531002, the Fundamental Research Funds for the Central Universities. The third author was also supported by the China Postdoctoral Science Foundation under Grant Number 2025M783146.



\end{document}